\documentclass[a4paper,11pt,reqno]{amsart}
\usepackage{amssymb}
\usepackage{enumitem}
\usepackage[alphabetic]{amsrefs}
\makeatletter
\@namedef{subjclassname@2010}{%
  \textup{2010} Mathematics Subject Classification}
\makeatother

\newtheorem{theorem}{Theorem}

\newtheorem{notminetheorem}{Theorem}

\newtheorem{corollary}{Corollary}[section]
\newtheorem{lemma}[corollary]{Lemma}
\newtheorem{proposition}[corollary]{Proposition}
\newtheorem{definition}[corollary]{Definition}

\newtheorem{remark}[corollary]{Remark}
\newtheorem*{theorem*}{Theorem}
\newtheorem*{corollary*}{Corollary}

\numberwithin{equation}{section}

\DeclareMathOperator{\Q}{\mathbb{Q}}
\DeclareMathOperator{\Z}{\mathbb{Z}}

\DeclareMathOperator{\R}{\mathbb{R}}
\DeclareMathOperator{\C}{\mathbb{C}}
\DeclareMathOperator{\N}{\mathbb{N}}
\DeclareMathOperator{\A}{\mathbb{A}}

\DeclareMathOperator{\bo}{\mathcal{O}}

\DeclareMathVersion{normal2}

\begin{document}

\baselineskip=17pt

\title[$S$-parts of values of univariate polynomials]{$S$-parts of values of univariate polynomials}

\author[M. Moreschi]{Maurizio Moreschi}
\address{Bachstraat 250 \\ 2324GS Leiden, Nederland}
\email{maurizio.moreschi.math@gmail.com}

\date{\today}

\begin{abstract}
    Let $S=\{p_1,\dots,p_s\}$ be a finite non-empty set of distinct prime numbers, let $f\in \Z[X]$ be a polynomial of degree $n\ge 1$, and let $S'\subseteq S$ be the subset of all $p\in S$ such that $f$ has a root in $\Z_p$. For any non-zero integer $y$, write $y=p_1^{k_1}\dots p_s^{k_s}y_0$, where $k_1,\dots,k_s$ are non-negative integers and $y_0$ is an integer coprime to $p_1,\dots,p_s$. We define the $f$-normalized $S$-part of $y$ by $[y]_{f,S}:=p_1^{k_1 r_{p_1,S}(f)}\dots p_s^{k_s r_{p_s,S}(f)}$, with $r_{p,S}(f)=1$ if $p\in S\setminus S'$ and $r_{p,S}(f)=R_{S'}(f)/R_{p}(f)$ if $p\in S'$, where $R_p(f)$ denotes the largest multiplicity of a root of $f$ in $\Z_p$ and $R_{S'}(f):=\max_{p\in S'} R_p(f)$. For positive real numbers $\varepsilon, B$ with $\varepsilon<R_{S'}(f)/n$, we consider the number $\widetilde{N}(f,S,\varepsilon,B)$ of integers $x$ such that $|x|\le B$ and $0<|f(x)|^{\varepsilon}\le [f(x)]_{f,S}$. We prove that if  $s':=\#S'\ge 1$, then $\widetilde{N}(f,S,\varepsilon,B)\asymp_{f,S,\varepsilon} B^{1-(n\varepsilon)/R_{S'}(f)}(\log B)^{s'-1}$ as $B\to \infty$. Moreover, if $f$
	has no multiple roots in $\Z_p$ for any $p\in S'$ and $s':=\#S'\ge 2$, then there exists a constant $C(f,S,\varepsilon)>0$ such that $\widetilde{N}(f,S,\varepsilon,B)\sim C(f,S,\varepsilon)\,B^{1-n\varepsilon}(\log B)^{s'-1}$ as $B\to \infty$.
	
\end{abstract}



\maketitle

\section{Introduction}

	Let $S$ be a finite non-empty set of primes. For any non-zero integer $y$, let
	\begin{equation*}
		|y|=\prod_{p} p^{v_p(y)}
	\end{equation*} 
	be the prime factorization of $|y|$, where $p$ runs over the set of all prime numbers. The \textit{$S$-part} of $y$ is defined by
	\begin{equation} \label{standard}
		[y]_S:=\prod_{p\in S} p^{v_p(y)}.
	\end{equation}  

      Motivated by previous work of Gross and Vincent (\cite{History}), Bugeaud, Evertse and Gy\H{o}ry proved in \cite{Preprint} that if $f\in \Z[X]$ is a polynomial of degree $n\ge 1$ without multiple roots, then for any $\delta>0$ and any $x\in\Z$ with $f(x)\ne 0$ one has 
	\begin{equation*}
 		[f(x)]_S\ll_{f,S,\delta} |f(x)|^{(1/n)+\delta}.
	\end{equation*} 

	Furthermore, the exponent $1/n$ is the best possible, in the sense that there exist infinitely many primes $p$ and infinitely many $x\in \Z$ such that 
	\begin{equation*}
	   f(x)\ne 0 \quad \text{and} \quad [f(x)]_{\{p\}}\gg_{f,p} |f(x)|^{1/n}.
	\end{equation*}

	If $\varepsilon\in (0,1/n)$, then the set of integers $x$ such that 
	\begin{equation*}
		0<|f(x)|^{\varepsilon}\le [f(x)]_S   
	\end{equation*}
	is infinite as soon as $f$ has a root in $\Z_p$ for some $p\in S$. More precisely, the following result for the asymptotic rate of the quantity
    	\begin{equation*} 
	    	 N(f,S,\varepsilon,B):=\#\{x\in\Z\,:\,|x|\le B,\,0< |f(x)|^{\varepsilon}\le [f(x)]_S\} 
    	\end{equation*}
   as $B\to \infty$ holds.
    
	\begin{notminetheorem}[\cite{Preprint}*{Theorem 2.3}] \label{notmine}
		Let $f(X)\in \Z[X]$ be a polynomial of degree $n\ge 1$ without multiple roots, let $S$ be a finite set of primes, and let $S'\subseteq S$ be the subset of all $p\in S$ such that $f$ has a root in $\Z_p$.
		Suppose that $s':=\#S'\ge 1$. Then, for any $\varepsilon\in (0,1/n)$ one has 
	\begin{equation*} 
 	N(f,S,\varepsilon,B)\asymp_{f,S,\varepsilon} B^{1-n\varepsilon} (\log B)^{s'-1} \quad\text{as  $B\to \infty$.}
	\end{equation*}
	\end{notminetheorem}
    
    Such result of Bugeaud, Evertse and Gy\H{o}ry is where the motivation for the present paper is to be found.
    
    The first main result of this paper appears already  (in a slightly less general formulation) in the author's master's thesis \cite{Thesis}, and it says that under the assumptions of theorem \ref{notmine} an exact asymptotics for $N(f,S,\varepsilon,B)$ as $B\to \infty$ is possible if and only if $s'\ge 2$. 
 
	\begin{theorem} \label{goalunivariate}
		Let $f(X)\in \Z[X]$ be a polynomial of degree $n\ge 1$, and let $\varepsilon \in (0,1/n)$. Also, let $S$ be a finite set of primes, and let $S'\subseteq S$ be the subset of all $p\in S$ such that $f$ has a root in $\Z_p$. Suppose that $f$ does not have multiple roots in $\Z_p$ for any $p\in S'$. We denote $s':=\#S'$. If $s'\ge 2$, then there exists a constant $C(f,S,\varepsilon)>0$ such that
	\begin{equation*}
		N(f,S,\varepsilon,B) \sim C(f,S,\varepsilon)\cdot B^{1-n\varepsilon}(\log B)^{s'-1} \quad \text{as $B\to \infty$}.
	\end{equation*}
        If $s'=1$, then $N(f,S,\varepsilon,B) \asymp_{f,S,\varepsilon} B^{1-n\varepsilon}$ as $B\to \infty$, but an exact asymptotics is not possible.
	\end{theorem}

    Going through the proof of theorem \ref{notmine} in \cite{Preprint}, it is not difficult to realize that the polynomial factor and the logarithmic factor in the asymptotic rate of $N(f,S,\varepsilon,B)$ as $B\to \infty$ have a very different nature. If $S'=\{p\}$, then the rate of $N(f,S,\varepsilon,B)$ as $B\to \infty$ is polynomial with exponent independent of the specific prime $p$, fact that is intimately related to the existence of an elementary asymptotic rate for $N(f,S,\varepsilon,B)$ as $B\to \infty$ in the case $\#S'\ge 2$. If $S'=\{p_1,\dots,p_{s'}\}$ with $s':=\# S'\ge 2$, then the logarithmic term that appears in the rate encodes information about the distribution of the numbers $p_1^{k_1}\dots p_{s'}^{k_{s'}}$ $((k_1,\dots,k_{s'})\in \Z_{\ge 0}^{s'})$ over the positive real line.

    If we allow the polynomial $f(X)\in \Z[X]$ to have multiple roots in $\Z_p$, then we can prove that in the case $S'=\{p\}$ one has
	    \begin{equation} \label{problem}
	       N(f,S,\varepsilon,B)\asymp_{f,S,\varepsilon} B^{1-(n\varepsilon)/R_p(f)} \quad \text{as $B\to \infty$,}
	    \end{equation}
    where $R_p(f)$ denotes the largest multiplicity of a root of $f$ in $\Z_p$. 
    
    The rate in (\ref{problem}) suggests that, in order to get an elementary asymptotic rate for $N(f,S,\varepsilon,B)$ as $B\to \infty$ when $\#S'\ge 2$, we need to require that the value $R_p(f)$ be the same for all $p\in S'$, in which case we say that $S$ is \textit{$f$-balanced}. The asymptotic rate of $N(f,S,\varepsilon,B)$ as $B\to \infty$ under this condition is a special case of our second main result.

    For $f$, $S$ and $S'$ as above, we introduce the notation
    \begin{equation*}
        R_{S'}(f):=\max_{p\in S'} R_p(f),
    \end{equation*}
    and for any $p\in S$,
   \begin{equation} \label{scaling}
            r_{p,S}(f):=\begin{cases} R_{S'}(f)/R_{p}(f) & \text{if $p\in S'$,} \\ 1 & \text{if $p\in S\setminus S'$.}\end{cases}
        \end{equation}
 
    The \textit{$f$-normalized $S$-part} of a non-zero integer $y$ is defined by
	\begin{equation} \label{balanced}
		[y]_{f,S}:=\prod_{p\in S} p^{v_p(y)r_{p,S}(f)}.
	\end{equation}  

    The second main result of this paper, the proof of which is given in section \ref{proofthm2} below, concerns the asymptotic rate of the quantity
  \begin{equation*} 
        \widetilde{N}(f,S,\varepsilon,B):=\#\{x\in \Z\,:\,|x|\le B,\,0<|f(x)|^{\varepsilon}\le [f(x)]_{f,S}\} 
    \end{equation*}
    as $B\to\infty$.
    
	\begin{theorem} \label{maintheorem}
		Let $f(X)\in \Z[X]$ be a polynomial of degree $n\ge 1$. Let $S$ be a finite set of primes, and let $S'\subseteq S$ be the subset of all $p\in S$ such that $f$ has a root in $\Z_p$.
		Suppose that $s':=\#S'\ge 1$. Then, for any $\varepsilon\in(0,R_{S'}(f)/n)$ one has
	\begin{equation*}
		\widetilde{N}(f,S,\varepsilon,B) \asymp_{f,S,\varepsilon} B^{1-(n\varepsilon)/R_{S'}(f)}(\log B)^{s'-1} \quad \text{as $B\to \infty$}.
	\end{equation*}
	\end{theorem}

     Definitions (\ref{standard}) and (\ref{balanced}) agree precisely when $S$ is $f$-balanced, in which case theorem \ref{maintheorem} provides the asymptotic rate of $N(f,S,\varepsilon,B)$ as $B\to \infty$. The condition of $S$ being $f$-balanced is trivially satisfied when $s'=1$ (which yields (\ref{problem})) or when $f$ has no multiple roots (which recovers theorem \ref{notmine}). Another remarkable case is when for all the primes $p$ in $S'$ one has that $p$  splits completely in a splitting field $K$ of $f$ over $\Q$ and that $\deg(f \mod p)=\deg f$. Since in this case $K$ embeds in $\Q_p$ for all $p\in S'$, all the roots of $f$ in $\C_p$ are $\Q_p$-rational, hence in $\Z_p$ (because of the condition on the degree of the reduction of $f$ modulo $p$), for all $p\in S'$. Theorem \ref{maintheorem} implies, therefore, the following corollary.

\begin{corollary*} \label{generalunivariate}
		Let $f(X)\in \Z[X]$ be a polynomial of degree $n\ge 1$ with splitting field $K$ over $\Q$ and leading coefficient $c_f$, let $S$ be a finite set of primes, and let $S'\subseteq S$ be the subset of all $p\in S$ such that $f$ has a root in $\Z_p$. Suppose that $s':=\# S'\ge 1$ and that all $p\in S'$  split completely in $K$ and do not divide $c_f$. Then, for any $\varepsilon\in(0,R(f)/n)$ one has 
	\begin{equation*}
 		N(f,S,\varepsilon,B)\asymp_{f,S,\varepsilon} B^{1-(n\varepsilon)/R(f)} (\log B)^{s-1} \quad\text{as $B\to \infty$},
	\end{equation*}
        where $R(f)$ denotes the largest multiplicity of a root of $f$ in $K$. 
\end{corollary*}

    In the proofs of theorems \ref{goalunivariate} and \ref{maintheorem}, we make use of two main technical tools. The first one is a formula, which we derive in section \ref{Igusasec}, for the Igusa local zeta functions of univariate polynomials. Such formula is, in fact, a special case of a formula given by Igusa in \cite{Igusabook} (last formula of page 123). However, in the case of univariate polynomials lots of technicalities can be avoided, and a fairly explicit formula can be obtained by direct computation. 
    
     The second tool is a careful asymptotic analysis of power sums indexed over sets of the form 
	\begin{equation} \label{Nsigma}
		\N_{\Sigma}:=\{q_1^{k_1}\dots q_s^{k_s}\,:\,(k_1,\dots,k_s)\in \Z_{\ge 0}^s \},
	\end{equation}
    where $\Sigma=\{q_1,\dots,q_s\}$ is a non-empty $\Q$-multiplicatively independent subset of $\R_{>1}$ (i.e. $\{\log q_1,\dots,\log q_s\}$ is a  $\Q$-linearly independent subset of $\R_{>0}$). 
    Section \ref{sums} is dedicated to the development of such tool. Modulo the omission, for the sake of brevity, of a few elementary details, the treatment is the same that can be found in sections $2.1-2.3$ of the author's master's thesis \cite{Thesis}.
    
   The techniques in this paper can be adapted to the similar problems considered in \cite{Preprint} in the context of decomposable forms. This leads to significant improvements on the corresponding results in \cite{Preprint}. We will present our results on decomposable forms in a subsequent paper.

\section{Igusa local zeta functions of univariate polynomials} \label{Igusasec}

Let $f\in \Z_p[X]$ be a polynomial of degree $n\ge 1$. We denote by $\mu_p$ the Haar probability measure on $\Z_p$ (cf. \cite{Koblitz}). The Igusa local zeta function of $f$ is the holomorphic function on the right half plane defined by
\[\zeta_{f,p}(s):=\int_{\Z_p} |f(x)|_p^s \,d\mu_p(x)\quad (\Re s>0).\]

 We know from \cite{Igusabook}*{Theorem 8.2.1} that $\zeta_{f,p}$ has a meromorphic continuation to the whole complex plane as rational function of $t=p^{-s}$. In this section, we recover, by direct computation, an explicit version of the formula given by Igusa in the proof of the above mention result. 

    For any $k\in \Z_{\ge 0}$, we denote
        \begin{equation*} 
            U_{p^k}(f):=\{x\in \Z_p\,:\,|f(x)|_p=p^{-k}\},
        \end{equation*}  
    so that we get the identity 
       \begin{equation} \label{Zseries}
            \zeta_{f,p}(s)=\sum_{k=0}^{\infty} \mu_p(U_{p^{k}}(f))\, t^k \quad (\Re s>0).
        \end{equation} 

  Let us first consider the case in which $f$ has no roots in $\Z_p$. Since the polynomial function $f:\Z_p\to \Z_p$ is continuous, so is also the composition $|f|_p:\Z_p\to p^{\Z_{\le 0}}\cup \{0\}$. This implies that the image of $|f|_p$ is compact. On the other hand, since $f$ has no zeros in $\Z_p$, the image of $|f|_p$ is also contained in the discrete subset $p^{\Z_{\le 0}}$, hence finite. We can then consider the maximum value of $v_p(f(x))$ for $x$ ranging $\Z_p$. Denoting such value by $u_p(f)$, we get the identity 
    \[ \zeta_{f,p}(s)=\sum_{k=0}^{u_p(f)} \mu_p(U_{p^{k}}(f))\, t^k\in \Z_{(p)}[t]\]
    on the right half $s$-plane, which provides a holomorphic continuation of $\zeta_{f,p}$ to $\C$ as a polynomial in $t=p^{-s}$.

    Suppose now that $f$ has roots in $\Z_p$. Let $\alpha_1,\dots,\alpha_l$ $(l\ge 1)$ be the list of distinct roots of $f$ in $\Z_p$, of multiplicities $r_1,\dots,r_l$ respectively. Then we have the factorization
    \begin{equation} \label{fact}
        f(X)=(X-\alpha_1)^{r_1}\dots (X-\alpha_l)^{r_l}g(X),
    \end{equation} 
    for some polynomial $g\in \Z_p[X]$ without zeros in $\Z_p$.
    
    Consistently with the introduction, we denote $R_p(f):=\max_{i}r_i$. Moreover, we introduce the quantities $\lambda_p(f)$ and $a_p(f)$ in the following definition.

   \begin{definition} \label{fundquant}
    	Let $f(X)\in \Z_p[X]$ be a polynomial factorizing as in (\ref{fact}). 
        \begin{enumerate}
            \item  We define the quantity $\lambda_p(f)$ to be the smallest non-negative integer $\lambda$ such that
    \begin{enumerate}
        \item $|\alpha_i-\alpha_j|_p\ge p^{-\lambda}$ for all $i,j\in\{1,\dots,l\}$ with $i\ne j$, and
        \item $|g(y+\alpha_i)|_p=|g(\alpha_i)|_p$ for all $i\in \{1,\dots,l\}$ and all $y\in \Z_p$ with $|y|_p<p^{-\lambda}$.
    \end{enumerate}
        \item  The quantity $a_p(f)$ is defined by
    	\[ a_p(f):=(r_1+\dots+r_l)\lambda_p(f)+R_p(f)+u_p(g)-1. \]   
        \end{enumerate}

    \end{definition}

    Now, let us denote
    \[ W:=\{y\in \Z_p\,:\,|y|_p<p^{-\lambda_p(f)}\},\quad W_i:=\alpha_i+W\;(i=1,\dots,l). \]
    
    Note that the sets $W_1,\dots,W_l$ are pairwise disjoint, for if there existed $x\in W_i\cap W_j$ for some $i\ne j$, then one would have $|x-\alpha_i|_p<p^{-\lambda_p(f)}$ and  $|x-\alpha_j|_p<p^{-\lambda_p(f)}$, leading to the contradiction $|\alpha_i-\alpha_j|_p<p^{-\lambda_p(f)}$. 
    
    This leads to the identity 
        \[ \zeta_{f,p}(s)=\sum_{i=1}^l \int_{W_i} |f(x)|_p^{s}\,d\mu_p(x)+\int_{W'} |f(x)|_p^{s}\,d\mu_p(x)\quad (\Re s>0), \]
    where $W':=\Z_p\setminus (W_1\cup\dots\cup W_l)$. 

    If $x\in W_i$, then we have $x=\alpha_i+y$ for some $y\in W$ and thus
    \[\begin{split} |f(x)|_p&=\Big(\prod_{j\ne i} |y+\alpha_i-\alpha_j|_p^{r_j}\Big) |y|_p^{r_i} |g(\alpha_i+y)|_p 
    \\
    &=\Big(\prod_{j\ne i} |\alpha_i-\alpha_j|_p^{r_j}\Big)|g(\alpha_i)|_p|y|_p^{r_i}, 
    \end{split}\]
    by definition of $\lambda_p(f)$ (and $W$).
    
    It follows that
    \[\begin{split}
        \int_{W_{i}} |f(x)|_p^{s}\,d\mu_p(x)&=\Big(\prod_{j\ne i} |\alpha_i-\alpha_j|_p^{r_j}\Big)^s|g(\alpha_i)|_p^s\int_W |y|_p^{r_i s}\,d\mu_p(y)
        \\
        &=\Big(\prod_{j\ne i} |\alpha_i-\alpha_j|_p^{r_j}\Big)^s |g(\alpha_i)|_p^s \sum_{\lambda=\lambda_p(f)+1}^{\infty} (1-p^{-1})p^{-\lambda} p^{-\lambda r_i s}
        \\
        &=\Big(\prod_{j\ne i} |\alpha_i-\alpha_j|_p^{r_j}\Big)^s|g(\alpha_i)|_p^s \frac{(1-p^{-1})(p^{-1-r_i s})^{\lambda_p(f)+1}}{1-p^{-1-r_is}} 
        \\
        &=\frac{(1-p^{-1})p^{-\lambda_p(f)-1}t^{k_i}}{1-p^{-1}t^{r_i}},
    \end{split}\]
    where
    \[k_i:=\sum_{j \ne i} r_j v_p(\alpha_i-\alpha_j)+v_p(g(\alpha_i))+r_i(\lambda_p(f)+1)\le a_p(f)+1. \]

    For the integral over $W'$, it is enough to note that for any $x\in W'$ one has $|x-\alpha_i|\ge p^{-\lambda_p(f)}\, \forall i\in\{1,\dots,l\}$ and $|g(x)|_p\ge p^{-u_p(g)}$, hence
    \[ |f(x)|_p\ge p^{-(r_1+\dots+r_l)\lambda_p(f)-u_p(g)}= p^{-(a_p(f)-R_p(f)+1)}. \]
  
    Putting everything together, we arrive to the identity
 \[ \zeta_{f,p}(s)=\sum_{i=1}^{l}\frac{(1-p^{-1})p^{-\lambda_p(f)-1} t^{k_i}}{1-p^{-1}t^{r_i}}+\sum_{k=0}^{a'_p(f)} \mu_p (U'_{p^k}(f))\, t^{k},\]
    on the right half $s$-plane, where $U'_{p^{k}}(f)$ denotes the set of all $x$ in $W'$ such that $|f(x)|_p=p^{-k}$ and $a'_p(f):=a_p(f)-R_p(f)+1\le a_p(f)$. This provides the desired meromorphic continuation of $\zeta_{f,p}$ to $\C$ as a rational function of $t=p^{-s}$.

    By (\ref{Zseries}) and the identity principle, we get that
    \begin{equation} \label{rationality}
        \sum_{k=0}^{\infty} \mu_p (U_{p^k}(f))\, t^{k}=\sum_{i=1}^{l}\frac{(1-p^{-1})p^{-\lambda_p(f)-1} t^{k_i}}{1-p^{-1}t^{r_i}}+\sum_{k=0}^{a'_p(f)} \mu_p (U'_{p^k}(f))\, t^{k}
    \end{equation}
    for all complex $t$ not in the set of poles \[\Omega=\{p^{1/r_i}\zeta_{r_i}^{j}\,:\,j=0,\dots,r_i-1,\,i=1,\dots,l\}\]
    (here $\zeta_{r_i}$ denotes a primitive $r_i$-th root of unity).

    The following proposition (cf. \cite{AnalComb}*{Theorem IV.9}) is an immediate consequence of (\ref{rationality}).
  
  	\begin{proposition} \label{ancom}
		Let $f(X)\in \Z_p[X]$ be a polynomial with $l\ge 1$ distinct roots in $\Z_p$. Then 
    \begin{enumerate}
        \item[(a)]	for any integer $k\ge a_p(f)+1$, one has
        \begin{equation*}
			\mu_p(U_{p^k}(f))\le \Big((1-p^{-1})p^{-\lambda_p(f)-1}\sum_{i=1}^l p^{k_i/r_i}\Big)\, p^{-k/R_p (f)};
		\end{equation*}
        \item[(b)] 	for any $i\in \{1,\dots,l\}$ such that $r_i=R_p(f)$ and any integer $k\ge a_p(f)+1$ with $k\equiv k_i \mod R_p(f)$, one has
        \begin{equation*}
			\mu_p(U_{p^{k}}(f)) \ge \Big((1-p^{-1})p^{-\lambda_p(f)-1}p^{k_i/R_p(f)}\Big)\,p^{-k/R_p (f)};
		\end{equation*} 
        \item[(c)]  in the case all the roots of $f$ in $\Z_p$ are simple, one has
        \begin{equation*}
			\mu_p(U_{p^k}(f))=\Big((1-p^{-1})p^{-\lambda_p(f)-1}\sum_{i=1}^l p^{k_i}\Big) p^{-k} \quad \forall k\ge a_p(f)+1.
		\end{equation*}
    \end{enumerate}

	\end{proposition}
	
    \begin{proof}
        Taking coefficients in (\ref{rationality}), we see that for all $k\ge a_p(f)+1$ one has
         \[ \mu_p(U_{p^k}(f))=(1-p^{-1})p^{-\lambda_p(f)-1}\sum_{i=1}^{l} \delta_i(k) p^{-(k-k_i)/r_i},  \]
        where
         \[\delta_i(k):=\begin{cases} 1 & \text{if $k\equiv k_i \mod r_i$,}  \\ 0 & \text{if $k\not\equiv k_i\mod r_i$.} \end{cases} \]
        All the three claims follow immediately.        
    \end{proof}

\section{Power sums over $\N_{\Sigma}$} \label{sums}

    Let $\Sigma=\{q_1,\dots,q_s\}$ be a non-empty $\Q$-multiplicatively independent subset of $\R_{>1}$. 	For each $h\in\N_{\Sigma}$ (cf. (\ref{Nsigma})), the numbers $v_{q_1}(h),\dots, v_{q_s}(h)\in\Z_{\ge 0}$ are uniquely determined by the writing $h =q_1^{v_{q_1}(h)} \dots q_s^{v_{q_s}(h)}$.
	
	In this section, we study the asymptotic behaviour as $L\to\infty$ of power sums of the form
	\begin{equation} \label{sumobjects}
	\underset{h\le L}{\sum_{h\in \N_{\Sigma}}} h^{\alpha} \qquad \text{or} \qquad \underset{h> L}{\sum_{h\in \N_{\Sigma}}} h^{-\alpha},
	\end{equation}
	where $\alpha\in \R_{>0}$.

    If $\Sigma=\{q\}$ for some $q\in \R_{>1}$, then these two sums are given, for all $L\in \R_{\ge 1}$, by the geometric sums
	\begin{equation} \label{pps11}
	 	\sum_{k=0}^{\lfloor\log_q L \rfloor} q^{k\alpha}=\frac{q^{\alpha(\lfloor\log_q L \rfloor+1)}-1}{q^{\alpha}-1}=\frac{q^{\alpha(1-\{\log_q L\})}}{q^{\alpha}-1}\, L^{\alpha} -\frac{1}{q^{\alpha}-1}
	\end{equation}
    and	
	\begin{equation}  \label{pps12}
		 \sum_{k=\lfloor\log_q L \rfloor+1}^{\infty} q^{-k\alpha}= \frac{1}{1-q^{-\alpha}}-\frac{1-q^{-\alpha(\lfloor\log_q L \rfloor+1)}}{1-q^{-\alpha}}=\frac{q^{\alpha\{\log_q L\}}}{q^{\alpha}-1}\,L^{-\alpha}
	\end{equation}
    respectively. 
    
    Note that
   \[ \liminf_{L\to\infty}\frac{1}{L^{\alpha}}\sum_{k=0}^{\lfloor\log_q L \rfloor} q^{k\alpha}=\frac{1}{q^{\alpha}-1}=\liminf_{L\to\infty}\frac{1}{L^{-\alpha}}\sum_{k=\lfloor\log_q L \rfloor+1}^{\infty} q^{-k\alpha} \]
 and   
   \[  \limsup_{L\to\infty}\frac{1}{L^{\alpha}}\sum_{k=0}^{\lfloor\log_q L \rfloor} q^{k\alpha}=\frac{q^{\alpha}}{q^{\alpha}-1}=\limsup_{L\to\infty}\frac{1}{L^{-\alpha}}\sum_{k=\lfloor\log_q L \rfloor+1}^{\infty} q^{-k\alpha},\]
    but the sequences that realize the first $\liminf$ (e.g. $L_m=q^{m-1/m}$) are exactly the sequences which realize the second $\limsup$ and, conversely, the sequences that realize the second $\liminf$ (e.g. $L_m=q^{m}$) are exactly the sequences which realize the first $\limsup$.
    
    We prove the following proposition for future purposes. 
    
	\begin{proposition} \label{oneprime}
		For any $q\in\R_{>1}$ and any $\alpha,\alpha'\in\R_{>0}$, one has
		\begin{equation*}	
			\liminf_{L\to \infty} \frac{1}{L^{\alpha}}\sum_{h\in \N_{\{q\}}} \min\{h^{\alpha},L^{\alpha+\alpha'}h^{-\alpha'}\}=\Big(1+\frac{\alpha}{\alpha'}\Big) \frac{q^{\alpha\alpha'/(\alpha+\alpha')}}{q^{\alpha}-1} \Big(\frac{\alpha'}{\alpha}\frac{q^{\alpha}-1}{q^{\alpha'}-1}\Big)^{\alpha/(\alpha+\alpha')}
		\end{equation*}
		and
		\begin{equation*}
			\limsup_{L\to \infty} \frac{1}{L^{\alpha}}\sum_{h\in \N_{\{q\}}} \min\{h^{\alpha},L^{\alpha+\alpha'}h^{-\alpha'}\}=\begin{cases} 1-\frac{1}{q^{\alpha}-1}+\frac{1}{q^{\alpha'}-1} & \alpha\ge \alpha', \\ 1-\frac{1}{q^{\alpha'}-1}+\frac{1}{q^{\alpha}-1} & \alpha\le \alpha'. \end{cases}
		\end{equation*}
	\end{proposition}

\begin{proof}
	From (\ref{pps11}) and (\ref{pps12}), we get
	\begin{equation*}  
		\sum_{h\in \N_{\{q\}}} \min\{h^{\alpha},L^{\alpha+\alpha'}h^{-\alpha'}\} = \bigg(\frac{q^{\alpha(1-\{\log_p L\})}}{q^{\alpha}-1}+\frac{q^{\alpha'\{\log_q L\}}}{q^{\alpha'}-1} \bigg) L^{\alpha} -\frac{1}{q^{\alpha}-1}. 
 	\end{equation*}
From the surjectivity of the map $\R\to [0,1)$, $L\mapsto \{\log_q L\}$, it follows that
	\begin{equation*}
		 \liminf_{L\to \infty}\frac{1}{L^{\alpha}}\sum_{h\in \N_{\{q\}}} \min\{h^{\alpha},L^{\alpha+\alpha'}h^{-\alpha'}\} =\inf_{u\in [0,1)} \mathcal{L}(u)
	\end{equation*}
	and
	\begin{equation*}	
		\limsup_{L\to \infty}\frac{1}{L^{\alpha}}\sum_{h\in \N_{\{q\}}} \min\{h^{\alpha},L^{\alpha+\alpha'}h^{-\alpha'}\}= \sup_{u\in [0,1)} \mathcal{L}(u).
	\end{equation*}
	where $\mathcal{L}:\R\to (0,\infty)$ is defined by 
	\begin{equation*}
		\mathcal{L}(u):=\frac{A^{1-u}}{A-1}+\frac{A^{\rho u}}{A^{\rho}-1} \qquad (A:=q^{\alpha},\; \rho:=\alpha'/\alpha).
	\end{equation*}

	The function $\mathcal{L}$ is convex, so it has a unique stationary point $u^*\in\R$, at which $\mathcal{L}$ assumes its global minimum over $\R$. A straightforward computation shows that 
	\begin{equation*}
		 u^*=\frac{1}{\alpha(1+\rho)}\bigg(\alpha-\log_q\Big(\frac{\rho(A-1)}{A^{\rho}-1}\Big)\bigg)\in (0,1),
	\end{equation*}
from which it follows that
        \begin{equation*}
            \begin{split}
 				\quad \inf_{u\in [0,1)}\mathcal{L}(u)&=\mathcal{L}(u^*)=  \Big( 1+\frac{1}{\rho} \Big)\frac{A}{A-1} A^{-1/(1+\rho)} \Big(\frac{\rho(A-1)}{A^{\rho}-1}\Big)^{1/(1+\rho)}, 
 				\\ 
 				\quad \sup_{u\in [0,1)}\mathcal{L}(u)&=\max\{\mathcal{L}(0),\mathcal{L}(1)\}=\begin{cases} 1-\frac{1}{A-1}+\frac{1}{A^{\rho}-1} & \text{if $\rho\le 1$} ,
 				\\
 				 1-\frac{1}{A^{\rho}-1}+\frac{1}{A-1} & \text{if $\rho\ge 1$}. 
 				 \end{cases} 
 			\end{split}
		\end{equation*}
\end{proof}

    \vspace{0,2cm}
    
    Let us now move to the case $\Sigma=\{q_1,\dots,q_s\}$, with $s\ge 2$. In this case, we want to show that the sums (\ref{sumobjects}) admit exact asymptotics as $L\to \infty$.

	\begin{definition}
		Let $\Sigma=\{q_1,\dots,q_s\}$ be a $\Q$-multiplicatively independent subset of $\R_{> 1}$, with $s\ge 2$. For any $\beta\in\R_{>1},\,t\in\Z_{\ge 0}$, we define 
	\begin{equation*}
		\mathcal{M}^{\beta}_t(\Sigma):=\Bigg\{\mathbf{x}\in \R^s\,:\,\begin{matrix} x_i\ge 0\quad\forall i\in\{1,\dots,s\}, \\
		 t< x_1 \log_{\beta} q_1 +\dots + x_s \log_{\beta} q_s \le t+1 
		 \end{matrix} \Bigg\}.
	\end{equation*}
		If $\beta=e$, then we drop the superscript. 
	\end{definition}

    The following lemma is the key result in the proof of the claimed exact asymptotics.
    
	\begin{lemma} \label{betatrick}
		Let $\Sigma=\{q_1,\dots,q_s\}$ be a $\Q$-multiplicatively independent subset of $\R_{>1}$, with $s\ge 2$. Then, there exists a constant $c(\Sigma)\in \R_{>0}$ such that for any $\beta\in \R_{>1}$ one has
	\begin{equation*} 
		\#(\Z^s\cap \mathcal{M}^{\beta}_t(\Sigma))= c(\Sigma)\cdot (\log \beta)^{s} t^{s-1}+ o_{\beta}(t^{s-1}) \quad \text{as $t\to \infty$}.
	\end{equation*}
	\end{lemma}

\begin{proof}
	For any $t\in \Z_{\ge 0}$, we can write $\mathcal{M}_t^{\beta}(\Sigma)=\mathcal{B}^{\beta}_{t+1}(\Sigma)\setminus \mathcal{B}^{\beta}_t(\Sigma)$, where 
	\begin{equation*}	
			\mathcal{B}^{\beta}_t(\Sigma):=\Bigg\{\mathbf{x}\in \R^s\,:\,\begin{matrix} x_i\ge 0\quad\forall i\in\{1,\dots,s\},\\  x_1 \log_{\beta} q_1 +\dots + x_s \log_{\beta} q_s \le t \end{matrix}\Bigg\}.
	\end{equation*}

	From \cite{Beta}*{Theorem 1}, it follows that there exist constants $c'(\Sigma), c''(\Sigma)\in\R_{>0}$ such that for any $\beta\in \R_{>1}$ one has
	\begin{equation*}
 		\#(\Z^s\cap \mathcal{B}^{\beta}_t(\Sigma))=c'(\Sigma) \cdot (\log\beta)^{s} t^{s}+c''(\Sigma)\cdot (\log \beta)^{s-1} t^{s-1}+o_{\beta}(t^{s-1}) 
	\end{equation*}
	as $t\to \infty$. The claim follows then with  $c(\Sigma):=c'(\Sigma)\cdot s$.
	\end{proof}

	For any $\beta>1$, the regions $\mathcal{M}_t^{\beta}$ $(t\in \Z_{\ge 0})$ give rise to a partition
	\begin{equation} \label{partition}
		\N_{\Sigma}\setminus \{1\}=\bigcup_{t=0}^{\infty} \big\{h\in\N_{\Sigma}\,:\,(v_{q_1}(h),\dots,v_{q_s}(h))\in \mathcal{M}_t^{\beta}(\Sigma)\big\},
	\end{equation}
	according to which we may split the power sums (\ref{sumobjects}). The partition (\ref{partition}) becomes finer and finer as $\beta\to 1^+$. The idea is then to estimate the summands, on each $\mathcal{M}_t^{\beta}(\Sigma)$, from below (resp. above) with the minimum (resp. the maximum) value they assume on $\mathcal{M}_t^{\beta}(\Sigma)$ (note that the ratio between these two values tends to $1$ as $\beta\to 1^+$). Combined with lemma \ref{betatrick}, this provides lower and upper bounds on the sums (\ref{sumobjects}), from which we deduce the asymptotic rates of the sums (\ref{sumobjects}) as $L\to \infty$. The existence of the desired exact asymptotics can then be proved by taking the limit $\beta\to 1^+$.
	
	
	The above paragraph describes the strategy for the proof of proposition \ref{pure} below. The following elementary lemma from discrete calculus is going to be necessary as well.
    
	\begin{lemma} \label{exppot}
		Let $\beta\in \R_{>1},\, \alpha\in\R_{>0},\,r\in\Z_{\ge 0}$. Then 
\[\begin{split}
&(a)\quad \sum_{t=0}^{T} \beta^{\alpha t} t^{r}=\frac{1}{\beta^\alpha-1}\,\beta^{\alpha(T+1)}T^r+ \bo_{\alpha,\beta}(\beta^{\alpha (T+1)} T^{r-1})\quad \text{as $T\to \infty$}, \\
&(b)\quad \sum_{t=T}^{\infty} \beta^{-\alpha t} t^{r}=\frac{1}{\beta^\alpha-1}\,\beta^{-\alpha(T+1)}T^r+ \bo_{\alpha,\beta}(\beta^{-\alpha (T+1)} T^{r-1})\quad \text{as $T\to \infty$}. 
\end{split}\]
	\end{lemma}
	\begin{proof}
        Both claims can be easily proved by induction on $r$, making use of the (discrete) summation by parts formula.
    \end{proof}

	\begin{proposition} \label{pure}
		Let $\Sigma=\{q_1,\dots,q_s\}$ be a $\Q$-multiplicatively independent subset of $\R_{>1}$, with $s\ge 2$. For any $\alpha\in\R_{>0}$, one has
		\begin{equation*}
			\begin{split} &(a)\quad \underset{h\le L}{\sum_{h\in\N_{\Sigma}}} h^{\alpha}\sim \frac{c(\Sigma)}{\alpha} L^{\alpha} (\log L)^{s-1} \quad \text{as $L\to \infty$,}
\\
& (b)\quad \underset{h>L}{\sum_{h\in\N_{\Sigma}}} h^{-\alpha}\sim \frac{c(\Sigma)}{\alpha} L^{-\alpha} (\log L)^{s-1} \quad \text{as $L\to \infty$,}
\end{split}
		\end{equation*}
    where $c(\Sigma)$ is the constant from lemma \ref{betatrick}.
	\end{proposition}

	\begin{proof}
	\begin{enumerate}
		\item[(a)] Estimating every $h\in \N_{\Sigma}$ such that $\log_{\beta} h\in \mathcal{M}_t^{\beta}(\Sigma)$ (for any $t\in\Z_{\ge 0}$) with $\beta^t$ from below and with $\beta^{t+1}$ from above, lemma \ref{exppot}(a) yields
	\begin{equation*}
 		\begin{split} \underset{h\le L}{\sum_{h\in\N_{\Sigma}}} h^{\alpha}   &\le 1+ \sum_{t=0}^{\lceil \log_{\beta} L \rceil-1} \beta^{\alpha(t+1)}\cdot\#(\Z^s\cap \mathcal{M}^{\beta}_t(\Sigma)) 
\\
&=1+\sum_{t=0}^{\lceil \log_{\beta} L \rceil-1} \beta^{\alpha(t+1)}\cdot \big(c(\Sigma)\cdot (\log \beta)^{s} t^{s-1}+ o_{\alpha,\beta}(t^{s-1})\big)
        \end{split}
	\end{equation*} 

	\begin{equation*}
	  \begin{split}
        &=c(\Sigma) \cdot (\log \beta)^{s}\bigg( \sum_{t=0}^{\lceil \log_{\beta} L \rceil-1} \beta^{\alpha(t+1)} t^{s-1} \bigg) +o_{\alpha,\beta}(L^{\alpha} (\log L)^{s-1})
\\
        &= \frac{c(\Sigma) (\log \beta)^{s}}{\beta^{\alpha}-1}\cdot \beta^{\alpha (1+\lceil \log_{\beta} L \rceil)} (\log_{\beta} L)^{s-1}+o_{\alpha,\beta}(L^{\alpha} (\log L)^{s-1})
\\
&\le \frac{\beta^{2\alpha} \log \beta}{\beta^{\alpha}-1} \cdot c(\Sigma)\cdot L^{\alpha} (\log L)^{s-1} +o_{\alpha,\beta}(L^{\alpha} (\log L)^{s-1}),
\end{split}
	\end{equation*} 
  from which it follows that
	\begin{equation*}
 		\limsup_{L\to \infty}\frac{1}{L^{\alpha} (\log L)^{s-1}} \underset{h\le L}{\sum_{h\in\N_{\Sigma}}} h^{\alpha} \le c(\Sigma) \cdot \lim_{\beta\to 1^+}\frac{\beta^{2\alpha} \log \beta}{\beta^{\alpha}-1} = \frac{c(\Sigma)}{\alpha}.   
\end{equation*}

	Similarly, one has 
	\begin{equation*} 
		\begin{split}
			\underset{h\le L}{\sum_{h\in\N_{\Sigma}}} h^{\alpha} &\ge \sum_{t=0}^{\lfloor \log_{\beta} L \rfloor-1} \beta^{\alpha t}\cdot\#(\Z^s\cap \mathcal{M}_t^{\beta}(\Sigma)) 
			\\
			&\ge\frac{\log \beta}{\beta^{\alpha}(\beta^{\alpha}-1)} \cdot c(\Sigma)\cdot L^{\alpha} (\log L)^{s-1} +o_{\alpha,\beta}(L^{\alpha} (\log L)^{s-1})
		\end{split}	
	\end{equation*}
	and thus
	\begin{equation*} 
		\liminf_{L\to \infty}\frac{1}{L^{\alpha} (\log L)^{s-1}} \underset{h\le L}{\sum_{h\in\N_{\Sigma}}} h^{\alpha} \ge c(\Sigma) \cdot \lim_{\beta\to 1^+}\frac{\log \beta}{\beta^{\alpha}(\beta^{\alpha}-1)} = \frac{c(\Sigma)}{\alpha}. 
	\end{equation*}

\item[(b)] The proof follows exactly the same lines as (a), using \ref{exppot}(b) in place of \ref{exppot}(a).
\end{enumerate}
\end{proof}

    \vspace{0,5cm}
    
    In the rest of this section, we give an application of propositions \ref{oneprime} and \ref{pure}. Proposition \ref{U} below is an important intermediate step in the proofs of theorems \ref{goalunivariate} and \ref{maintheorem}.
    
   Let $f\in \R[X]$ be a polynomial of degree $n\ge 1$. For any $B,M\in \R_{>0}$, we introduce the notation
	\begin{equation*}
	    V_f(B,M):=\{x\in \R\,:\,|x|\le B,\, |f(x)|\le M\}.
	\end{equation*}

    Let also $\gamma\in \R_{>0}$, $\sigma\in\R_{<0}$, $\varepsilon\in (0,-1/(\sigma n))$, and let $\Sigma=\{q_1,\dots,q_s\}$ ($s\ge 1$) be a $\Q$-multiplicative independent subset of $\R_{>1}$. Propositions \ref{oneprime} and \ref{pure}, together with a careful use of the polynomial growth, provide a precise description of the asymptotic behaviour of the quantity
	\begin{equation*} 
		 \mathcal{U}(f,\Sigma,\varepsilon,B,\gamma,\sigma):=\sum_{h\in\N_{\Sigma}} \mu_{\infty} (V_f(B,(\gamma h)^{1/\varepsilon}))\cdot h^{\sigma}
	\end{equation*} 
    as $B\to\infty$, where $\mu_{\infty}$ denotes the Lebesgue measure on $\R$. 
	
	In the case $\Sigma=\{q\}$, we introduce the following auxiliary notation.
    \begin{definition} \label{lambdas}
        For any $n\in\Z_{\ge 1}$, $\sigma\in \R_{<0}$, $q\in \R_{>1}$, $\varepsilon\in (0,-1/(\sigma n))$, we denote
	\begin{equation*} 
		 \begin{split}
 \lambda^-(n,\sigma,q,\varepsilon)&:=-\frac{1}{\sigma n\varepsilon} \frac{q^{-\sigma(1+\sigma n\varepsilon)}}{q^{1/(n\varepsilon)+\sigma}-1}  \Big(-\frac{\sigma}{1/(n\varepsilon)+\sigma}\,\frac{q^{1/(n\varepsilon)+\sigma}-1}{q^{-\sigma}-1}\bigg)^{1+\sigma n\varepsilon}, 
\\
\lambda^+(n,\sigma,q,\varepsilon)&:=\begin{cases} 1-\frac{1}{q^{1/(n\varepsilon)+\sigma}-1}+\frac{1}{q^{-\sigma}-1}  & \varepsilon\le -\frac{1}{2\sigma n}, \\   1-\frac{1}{q^{-\sigma}-1}+\frac{1}{q^{1/(n\varepsilon)+\sigma}-1}  & \varepsilon\ge -\frac{1}{2\sigma n}. \end{cases} \end{split}
	\end{equation*}
    \end{definition}

	\begin{proposition} \label{U}
		Let $f\in \R[X]$ be a polynomial of degree $n\ge 1$ and leading coefficient $c_f$. Let also $\gamma\in \R_{>0}$, $\sigma\in\R_{<0}$, $\varepsilon\in (0,-1/(\sigma n))$, and let $\Sigma=\{q_1,\dots,q_s\}$ $(s\ge 1)$ be a $\Q$-multiplicative independent subset of $\R_{>1}$. 
		\begin{enumerate}
			\item[(a)] If $\Sigma=\{q\}$, then one has
			\begin{equation*}
				\begin{split}
 					\liminf_{B\to \infty}\frac{\mathcal{U}(f,\{q\},\varepsilon,B,\gamma,\sigma)}{B^{1+\sigma n\varepsilon}} &=2\cdot \lambda^{-}(n,\sigma,q,\varepsilon) \cdot |c_f|^{\sigma\varepsilon}\gamma^{-\sigma},   \\
					 \limsup_{B\to \infty}\frac{\mathcal{U}(f,\{q\},\varepsilon,B,\gamma,\sigma)}{B^{1+\sigma n\varepsilon}} &=2\cdot \lambda^{+}(n,\sigma,q,\varepsilon) \cdot |c_f|^{\sigma\varepsilon} \gamma^{-\sigma}.
 				\end{split} 
			\end{equation*}

			\item[(b)] If $s\ge 2$, then
			\begin{equation*}
				\mathcal{U}(f,\Sigma,\varepsilon,B,\gamma,\sigma)\sim 2\cdot c(\Sigma) \cdot\frac{|c_f|^{\sigma\varepsilon}\gamma^{-\sigma}}{-\sigma(1+\sigma n\varepsilon)}\cdot B^{1+\sigma n\varepsilon} (\log B)^{s-1}
			\end{equation*}
            as $B\to \infty$, where $c(\Sigma)$ is the constant from lemma \ref{betatrick}.
		\end{enumerate}
	\end{proposition}

\begin{proof}
	For any $\delta\in (0,1/2)$ there exists $B_{\delta}>1$ such that for all $x\in \R$ with $|x|\ge B_{\delta}$ one has \[ (1-\delta)|c_f||x|^n \le |f(x)| \le (1+\delta) |c_f||x|^n.\]

	It follows that for any $\delta\in (0,1/2)$ one has
	\begin{equation*}
		\begin{split} 
		\liminf_{B\to \infty}\frac{\mathcal{U}_{\delta}(f,\Sigma,\varepsilon,B,\gamma,\sigma)}{B^{1+\sigma n\varepsilon}(\log B)^{s-1}} &\le \liminf_{B\to \infty} \frac{\mathcal{U}(f,\Sigma,\varepsilon,B,\gamma,\sigma)}{B^{1+\sigma n\varepsilon}(\log B)^{s-1}} \\ &\le \liminf_{B\to \infty}\frac{\mathcal{U}_{-\delta}(f,\Sigma,\varepsilon,B,\gamma,\sigma)}{B^{1+\sigma n\varepsilon}(\log B)^{s-1}}
		\end{split}
	\end{equation*}
	and
	\begin{equation*}
		\begin{split} 
		\limsup_{B\to \infty}\frac{\mathcal{U}_{\delta}(f,\Sigma,\varepsilon,B,\gamma,\sigma)}{B^{1+\sigma n\varepsilon}(\log B)^{s-1}} &\le \limsup_{B\to \infty} \frac{\mathcal{U}(f,\Sigma,\varepsilon,B,\gamma,\sigma)}{B^{1+\sigma n\varepsilon}(\log B)^{s'-1}} \\ &\le \limsup_{B\to \infty}\frac{\mathcal{U}_{-\delta}(f,\Sigma,\varepsilon,B,\gamma,\sigma)}{B^{1+\sigma n\varepsilon}(\log B)^{s-1}},
		\end{split}
	\end{equation*}		
	where
	\begin{equation*}
		\mathcal{U}_{\pm \delta}(f,\Sigma,\varepsilon,B,\gamma,\sigma):= \sum_{h\in\N_{S}}2\min\big\{B,((1\pm \delta)^{-\varepsilon}|c_f|^{-\varepsilon}\gamma h)^{1/(n\varepsilon)}\big\}\cdot h^{\sigma}.
	\end{equation*} 

	On the other hand, one has
	\begin{equation*}
 		\begin{split} 
	 		\liminf_{B\to \infty}\frac{\mathcal{U}_{\pm\delta}(f,\{q\},\varepsilon,B,\gamma,\sigma)}{B^{1+\sigma n\varepsilon}}&=2\cdot\lambda^{-}(n,\sigma,q,\varepsilon)\cdot (1\pm\delta)^{-\sigma\varepsilon} |c_f|^{\sigma\varepsilon} \gamma^{-\sigma}, 
	 		\\
	 		\limsup_{B\to \infty}\frac{\mathcal{U}_{\pm\delta}(f,\{q\},\varepsilon,B,\gamma,\sigma)}{B^{1+\sigma n\varepsilon}}&=2\cdot\lambda^{+}(n,\sigma,q,\varepsilon)\cdot (1\pm\delta)^{-\sigma\varepsilon} |c_f|^{\sigma\varepsilon} \gamma^{-\sigma}, 
	 	\end{split}
	\end{equation*}
	by proposition \ref{oneprime}, and 
	\begin{equation*}
 		\lim_{B\to \infty}\frac{\mathcal{U}_{\pm\delta}(f,\Sigma,\varepsilon,B,\gamma,\sigma)}{B^{1+\sigma n\varepsilon}(\log B)^{s-1}}=2\cdot c(\Sigma)\cdot\frac{(1\pm\delta)^{-\varepsilon} |c_f|^{\sigma\varepsilon} \gamma^{-\sigma}}{-\sigma(1+\sigma n\varepsilon)}
	\end{equation*}
	when $s\ge 2$, by proposition \ref{pure}.
 
	Both claims $(a)$ and $(b)$ follow now by taking the limit $\delta\to 0^+$.
\end{proof}

\section{Proof of theorem \ref{maintheorem}} \label{proofthm2} 

    Let $f\in \Z[X]$ be a polynomial of degree $n\ge 1$, let $S$ be a finite non-empty set of primes, and let $S'\subseteq S$ be the subset of all $p$ in $S$ such that $f$ has a root in $\Z_p$. 
    The numbers $r_{p,S}(f)$ ($p\in S$) are defined as in (\ref{scaling}). Let also $\varepsilon\in (0,R_{S'}(f)/n)$ and $\gamma,B\in \R_{>0}$. Adjusting an idea from \cite{LiuThesis}, we interpret the set of integers $x$ with $|x|\le B$ and $0<|f(x)|^{\varepsilon}\le\gamma \cdot [f(x)]_{f,S}$ as the set of integer points in the subset 
 \[ \A(f,S,\varepsilon,B,\gamma):=\Big\{(x_v)_v\in [-B,B]\times \widehat{\Z}\, : \, 0<|f(x_{\infty})|^{\varepsilon}\prod_{p\in S}|f(x_p)|_p^{r_{p,S}(f)}\le \gamma   \Big\} \]
 of $\R\times \widehat{\Z}$, with $\Z$ embedded diagonally in $\R\times \widehat{\Z}$. Therefore
 \[ \widetilde{N}(f,S,\varepsilon,B)=\#(\Z\cap \A(f,S,\varepsilon,B,1)). \]

   For any $h\in \N_{S}$, let $\A_h(f,S,\varepsilon,B,\gamma)\subseteq \A(f,S,\varepsilon,B,\gamma)$ be the subset of all $(x_v)_v$ in $\A(f,S,\varepsilon,B,\gamma)$ such that $|f(x_p)|_p=p^{-v_p(h)}$ for all $p\in S$. These sets are all pluri-rectangles, because of the decomposition
         \begin{equation} \label{proddec}
             \A_h(f,S,\varepsilon,B,\gamma)=V_f(B,(\gamma\xi_f(h))^{1/\varepsilon})\times\prod_{p\in S} U_{p^{v_p(h)}}(f)\times\prod_{p\not\in S}\Z_p,
         \end{equation} 
          where
        \[\xi_f(h):=\prod_{p\in S} p^{r_{p,S}(f)v_p(h)}. \]
    Denoting by $\mu:=\bigotimes_{v} \mu_v$ ($v$ running over all places of $\Q$) the product measure on $\R\times \widehat{\Z}$, we get thus
        \begin{equation} \label{meash}
            \mu(\A_h(f,S,\varepsilon,B,\gamma))=\mu_{\infty}(V_f(B,(\gamma\xi_f(h))^{1/\varepsilon}))\prod_{p\in S}\mu_{p}(U_{p^{v_p(h)}}(f))
        \end{equation}
    for all $h\in\N_{S}$.

    For any $h\in \N_{S}$, we can write $h=h_0h'$ for some $h_0\in \N_{S\setminus S'}$, $h'\in \N_{S'}$. It follows from (\ref{proddec}) that $\A_h(f,S,\varepsilon,B,\gamma)=\emptyset$ unless $h_0$ is a divisor of   
    \[ H_{S}(f):=\prod_{p\in S\setminus S'} p^{u_p(f)}.\]
    
       This gives us the disjoint union decomposition
        \begin{equation} \label{disjoint}
            \A(f,S,\varepsilon,B,\gamma)=\bigcup_{h_0|H_S(f)}\bigcup_{h'\in \N_{S'}} \A_{h_0h'}(f,S,\varepsilon,B,\gamma).
        \end{equation}
    
    Furthermore, we see from (\ref{meash}) that for any $h_0\in \N_{S\setminus S'}$, $h'\in \N_{S'}$ one has
      \begin{equation} \label{meashh0}
          \mu(\A_{h_0h'}(f,S,\varepsilon,B,\gamma))=C_{h_0}(f)\cdot \mu(\A_{h'}(f,S',\varepsilon,B,\gamma h_0)),
      \end{equation}
  where we denote
	\[C_{h_0}(f):=\prod_{p\in S\setminus S'}\mu_p(U_{p^{v_p(h_0)}}(f)). \]
	
      From (\ref{disjoint}) and (\ref{meashh0}), we finally get
        \[ \begin{split}
        \mu(\A(f,S,\varepsilon,B,\gamma))&=\sum_{h_0|H_S(f)} \sum_{h'\in \N_{S'}} \mu(\A_{h_0h'}(f,S,\varepsilon,B,\gamma))
        \\
        &=\sum_{h_0|H_S(f)} \sum_{h'\in \N_{S'}} C_{h_0}(f) \cdot \mu(\A_{h'}(f,S,\varepsilon,B,\gamma h_0))
        \\
        &=\sum_{h_0|H_S(f)} C_{h_0}(f)\cdot\mu(\A(f,S',\varepsilon,B,\gamma h_0)).
        \end{split}\]
		
		The asymptotic rate of $\mu(\A(f,S,\varepsilon,B,\gamma))$ as $B\to\infty$ is obtained by combining the results from sections \ref{Igusasec} and \ref{sums}.

	\begin{proposition} \label{leadinggen}
	    Let $f\in \Z[X]$ be a polynomial of degree $n\ge 1$, let $S$ be a finite non-empty set of primes, and let $S'\subseteq S$ be the subset of all $p$ in $S$ such that $f$ has a root in $\Z_p$. Suppose that $s':=\#S'\ge 1$. Then, for any $\varepsilon\in (0,R_{S'}(f)/n)$ and any $\gamma\in \R_{>0}$ one has
            \[\mu(\A(f,S,\varepsilon,B,\gamma))\asymp_{f,S,\varepsilon} \gamma^{1/R_{S'}(f)}\cdot B^{1-n\varepsilon/R_{S'}(f)}(\log B)^{s'-1}  \quad \text{as }  B\to \infty,\]
        with implied constants independent of $\gamma$.
    \end{proposition}	
		
		\begin{proof}
		    Because of the above discussion, we may assume $S=S'$ without loss of generality. From proposition \ref{ancom} (points $(a)$ and $(b)$), it follows that there exist constants $C>0$ and $h^*\in \N_{S}$ such that  
		    \begin{equation} \label{e1}
		        \prod_{p\in S} \mu_{p}(U_{p^{v_p(h)}}(f))\le C\cdot \xi_f(h)^{-1/R_S(f)}\quad \forall h\in \N_{S} 
		    \end{equation}
	        and
		    \begin{equation} \label{e2}
		       \prod_{p\in S} \mu_{p}(U_{p^{v_p(h^*\widetilde{h})}}(f))\ge \frac{1}{C} \cdot \xi_f(h^*\widetilde{h})^{-1/R_S(f)} \quad \forall \widetilde{h}\in \N_{\widetilde{S}},
		    \end{equation} 
		    where $\widetilde{S}:=\{p^{R_{p}(f)}\,:\,p\in S\}$.
		        
		        Note that the rule $h\mapsto\xi_f(h)^{1/R_S(f)}$ yields a bijection $\N_{S}\to \N_{\Sigma}$, with $\Sigma:=\{p^{1/R_p(f)}\,:\,p\in S\}$. Together with (\ref{e1}), this tells us that
		    \[ \begin{split}
		        \mu(\A(f,S,\varepsilon,B,\gamma))&=\sum_{h\in \N_S} \mu_{\infty}(V_f(B,(\gamma\xi_f(h))^{1/\varepsilon})\prod_{p\in S} \mu_{p}(U_{p^{v_p(h)}}(f))
		        \\
		        &\le C \sum_{\mathfrak{h}\in \N_{\Sigma}} \mu_{\infty}(V_f(B,(\gamma^{1/R_S(f)}\mathfrak{h})^{R_S(f)/\varepsilon})) \mathfrak{h}^{-1}
		        \\
		        &=C\cdot \mathcal{U}(f,\Sigma,\varepsilon/R_S(f),B,\gamma^{1/R_{S}(f)},-1).
		    \end{split} \]
		    
		        Similarly, the fact that the rule $h\mapsto\xi_f(h)^{1/R_S(f)}$ yields a bijection $\N_{\widetilde{S}}\to \N_{S}$, together with (\ref{e2}), give us
		    
		    	    \[ \begin{split}
		        \mu(\A(f,S,&\varepsilon,B,\gamma))\ge\sum_{\widetilde{h}\in \N_{\widetilde{S}}} \mu_{\infty}(V_f(B,(\gamma\xi_f(h^*\widetilde{h}))^{1/\varepsilon})\prod_{p\in S} \mu_{p}(U_{p^{v_p(h)}}(f))
		        \\
		        &\ge \frac{1}{C} \sum_{\mathfrak{h}\in \N_{S}} \mu_{\infty}(V_f(B,((\gamma\xi_f(h^*))^{1/R_S(f)}h)^{R_S(f)/\varepsilon})) \xi_f(h^*)^{-1/R_S(f)} h^{-1}
		        \\
		        &=\frac{\xi_f(h^*)^{-1/R_S(f)}}{C}\cdot \mathcal{U}(f,S,\varepsilon/R_S(f),B,(\gamma\xi_f(h^*))^{1/R_{S}(f)},-1).
		    \end{split} \]
		   
		   The claim follows now directly from proposition \ref{U}.
		\end{proof}
	
	\vspace{0,5cm}
        
        In order to deduce theorem \ref{maintheorem} from proposition \ref{leadinggen}, what is left to show is that the difference
            \begin{equation} \label{canderr}
                |\#(\Z\cap\A(f,S,\varepsilon,B,\gamma))-\mu(\A(f,S,\varepsilon,B,\gamma))|
            \end{equation} 
	    is negligible with respect to $\mu(\A(f,S,\varepsilon,B,\gamma))$ as $B\to \infty$. In fact, in a similar fashion to the proof of \cite{LiuThesis}*{Proposition 1.4.6}, we show that (\ref{canderr}) is bounded from above by a power of $\log B$ as $B\to \infty$. 

	\begin{lemma} \label{archimedean}
		Let $f(X)\in \R[X]$. For any $a\in\R$ and any $\lambda,B,M\in\R_{>0}$, one has
		\begin{equation*}
 			\bigg|\#((a+\lambda\Z)\cap V_f(B,M))-\frac{\mu_{\infty}(V_f(B,M))}{\lambda}\bigg|\le 2(n+1).
		\end{equation*}
	\end{lemma}

	\begin{proof}
		Note that the set $V_f(B,M)$ can be written as a disjoint union of $N\le n+1$ intervals $I_1,\dots, I_N$. Therefore
		\begin{equation*}
 			\begin{split} \bigg|\#((a+\lambda\Z)\cap V_f(B,M))&-\frac{\mu_{\infty}(V_f(B,M)}{\lambda}\bigg|
 			\\
 			 &\le \sum_{j=1}^{N} \bigg|\#((a+\lambda\Z)\cap I_j)-\frac{\mu_{\infty}(I_j)}{\lambda}\bigg| 
 			 \\
 			 &=\sum_{j=1}^{N} \bigg|\#\Big(\Z\cap \Big(-\frac{a}{\lambda}+\frac{1}{\lambda} I_j\Big)\Big)-\mu_{\infty}\Big(-\frac{a}{\lambda}+\frac{1}{\lambda} I_j\Big)\bigg| 
 			 \\
 			 & \le 2N 
 			 \\
 			 &\le 2(n+1).
 			 \end{split}
 		\end{equation*}
	\end{proof}

	\begin{proposition} \label{errorbound}
		Let $f\in \Z[X]$ be a polynomial of degree $n\ge 1$, let $S$ be a finite set of primes, and let $S'$ denote the subset of all $p\in S$ such that $f$ has a root in $\Z_p$. Denote the cardinality of $S'$ by $s'$. Then, one has
		\begin{equation*}
 			\big|\#(\Z\cap \A(f,S,\varepsilon,B,\gamma))-\mu(\A(f,S,\varepsilon,B,\gamma))\big| \ll_{f,S} (\log B)^{s'}\quad \text{as $B\to\infty$},
		\end{equation*}
	 with implied constant independent of $\varepsilon$ and $\gamma$.
	\end{proposition}

	\begin{proof} 
		Let $K$ be a splitting field of $f$ over $\Q$ and let 
  		\begin{equation*} 
      		f(X)=c \,(X-\alpha_1)\dots (X-\alpha_n).
  		\end{equation*}
		be the factorization of $f$ in $K[X]$, where $c\in \Z_{\ne 0}$ denotes the leading coefficient of $f$ and $\alpha_1,\dots,\alpha_n$ are the (not necessarily distinct) roots of $f$ in $K$.

    Let now $p\in S$, and let $\mathfrak{p}$ be a prime of $K$ above $p$. Since $K$ is Galois over $\Q$, the ramification index $e(\mathfrak{p}/p)$ does not depend on the particular choice of $\mathfrak{p}$, so we can denote it by $e_p$ without creating any confusion. We also denote by $\alpha_{pj}$ the image of $\alpha_j$ under the embedding $K\hookrightarrow K_{\mathfrak{p}}$, for any $j\in\{1,\dots,n\}$. Recall that if $\varpi$ is a local uniformizer parameter for $K_p$, then one has $|\varpi|_p = p^{1/e_p}$ (cf. \cite{Neukirch}).


		Let us fix $h\in \N_{S}$ for the moment, and let   
	$\mathcal{J}_0$ denote the set of all pairs $(p,j)$ with $p\in S$ and $j\in \{1,\dots,n\}$. Moreover, we denote by $\mathcal{K}_h(B)$ the subset of all tuples $\mathbf{k}\in \Z^{\mathcal{J}_0}$ such that the set
			\begin{equation*}
			\mathbb{V}_h(\mathbf{k};B):=\Bigg\{(x_v)_v\in\A_h(f,S,\varepsilon,B,\gamma)\,:\,  \begin{matrix} |x_p-\alpha_{pj}|_{p}=p^{-k_{pj}/e_p} \\ \forall (p,j) \in\mathcal{J}_0 \end{matrix}\Bigg\}
		\end{equation*}
	is non-empty.

	We get then the disjoint union of non-empty sets
		\begin{equation*}
 			\A_h(f,S,\varepsilon,B,\gamma)=\bigcup_{\mathbf{k}\in \mathcal{K}_h(B)}\mathbb{V}_h(\mathbf{k};B).
		\end{equation*}

        For any $\mathbf{\sigma}=(\sigma_p)_p\in \mathfrak{S}_n^S$, we consider the subset $\mathcal{K}^{\sigma}_h(B)\subseteq \mathcal{K}_h(B)$ of all $\mathbf{k}\in \mathcal{K}_h(B)$ with $k_{p\sigma_p(1)}\le \dots \le k_{p\sigma_p(n)}$ for all $p\in S$. 
        
        Pick $(\sigma_p)_p\in \mathfrak{S}_n^S$ such that $\mathcal{K}^{\sigma}_h(B)\ne\emptyset$, and let $\mathbf{k}\in \mathcal{K}_h^{\sigma}(B)$, $(x_v)_v\in \mathbb{V}_h(\mathbf{k};B)$. For some indexes $1=j_1<\dots <j_t\le n$, one has 
        \[ k_{p\sigma_p(j_1)}< k_{p\sigma_p(j_2)}<\dots< k_{p\sigma_p(j_t)}\]
        and
        \[ \begin{cases} k_{p\sigma_p(j)}=k_{p\sigma_p(j_l)} & l\in \{1,\dots,t-1\},\,j\in\{j_l,\dots,j_{l+1}-1\},\\
        k_{p\sigma_p(j)}=k_{p\sigma_p(j_t)}& j\in \{j_t,\dots,n\}.\end{cases}\] 
      
        For all $l\in \{1,\dots,t-1\}$ we have then $|x_p-\alpha_{p\sigma_p(j_l)}|_p>|x_p-\alpha_{p\sigma_p(j_{l+1})}|_p$, which implies
       \[  |\alpha_{p\sigma_p(j_l)}-\alpha_{p\sigma_p(j_{l+1})}|_p=|x_p-\alpha_{p\sigma_p(j)}|_p=p^{-k_{p\sigma_p(j_l)}}.  \]
    This shows that the components
        \[ k_{p\sigma_p(j_l)}=v_p(\alpha_{p\sigma_p(j_l)}-\alpha_{p\sigma_p(j_{l+1})})\quad (l\in\{1,\dots,t-1\})\]
    of $\mathbf{k}$ are univocally determined by $\alpha_{p1},\dots,\alpha_{pn}$. On the other hand, from the condition
    \[ (n-j_t+1)k_{p\sigma_p(j_t)}+\sum_{l=1}^{t-1}(j_{l+1}-j_{l})k_{p\sigma_p(j_l)}=e_pk_p-v_p(c),  \]
    we see that $k_{p\sigma_p(j_t)}$, hence the whole $\mathbf{k}$, is univocally determined by $\alpha_{p1},\dots,\alpha_{pn}$ as well.
    
   It follows that
    \[ \#\mathcal{K}^{\mathbf{\sigma}}_h(B)\le 2^{n-1}\quad \forall \mathbf{\sigma}\in \mathfrak{S}_n^S \]
    and thus
    \[ \#\mathcal{K}_h(B)\le 2^{n-1}n!^s.\]

		Let now $\mathbf{k}\in \mathcal{K}_h(B)$. For each $\mathcal{J}\subseteq \mathcal{J}_0$, we consider the subset $\mathbb{V}_h(\mathbf{k},\mathcal{J};B)$ of $\A_h(f,S,\varepsilon,B,\gamma)$ defined by the inequalities
		\begin{equation*} 
				\begin{cases}  
					|x_p-\alpha_{pj}|_{p}< p^{-k_{pj}/e_p} & \forall (p,j)\in \mathcal{J}, 
					\\
					 |x_p-\alpha_{pj}|_{p}\le p^{-k_{pj}/e_p} &  \forall (p,j)\in \mathcal{J}_0\setminus\mathcal{J}.
				 \end{cases}
		\end{equation*}

	Since 
		\begin{equation*}
			\mathbb{V}_h(\mathbf{k};B)=\mathbb{V}_h(\mathbf{k},\emptyset;B)\setminus \underset{\# \mathcal{J}=1}{\bigcap_{\mathcal{J}\subseteq \mathcal{J}_0}}\mathbb{V}_h(\mathbf{k},\mathcal{J};B),
		\end{equation*}
	the inclusion-exclusion principle yields
	\begin{equation} \label{inclexcl1}
			\mu(\mathbb{V}_h(\mathbf{k};B))=\sum_{l=0}^{ns}(-1)^l\underset{\#\mathcal{J}=l}{\sum_{\mathcal{J}\subseteq \mathcal{J}_0}}\mu(\mathbb{V}_h(\mathbf{k},\mathcal{J};B))
	\end{equation}
	and
	\begin{equation} \label{inclexcl2}
			\#(\Z\cap\mathbb{V}_h(\mathbf{k};B))=\sum_{l=0}^{ns}(-1)^l\underset{\#\mathcal{J}=l}{\sum_{\mathcal{J}\subseteq \mathcal{J}_0}}\#(\Z\cap \mathbb{V}_h(\mathbf{k},\mathcal{J};B)).
	\end{equation}

	If the set $\mathbb{V}_h(\mathbf{k},\mathcal{J};B)$ is non-empty, then it is of the form 
	\begin{equation*}
		V_f(B,M)\times \prod_{p\in S}(\alpha_p+p^{\kappa_p}\Z_p)
	\end{equation*}
	for some $M\in \R_{>0}$, $\kappa_p\in \Z_{\ge 0}$, $\alpha_p\in \{0,\dots,p^{\kappa_p}-1\}$ ($p\in S$), with 
	\begin{equation*}
		\kappa_p\ge \max_{j\in \{1,\dots,n\}} \frac{k_{pj}}{e_p}.
	\end{equation*}

	Together with the Chinese remainder theorem, this implies that for some $\alpha\in \{0,\dots,h-1\}$ one has 
	\[\Z\cap\mathbb{V}_h(\mathbf{k},\mathcal{J};B)=(\alpha+\widehat{h}\Z)\cap V_f\big(B,M\big),\quad
		\widehat{h}:=\prod_{p\in S} p^{\kappa_p}.\]

    From lemma \ref{archimedean}, it follows then that
	\[	\big|\#(\Z\cap\mathbb{V}_h(\mathbf{k},\mathcal{J};B))-\mu(\mathbb{V}_h(\mathbf{k},\mathcal{J};B))\big|\le 2(n+1),\]
	which, combined with (\ref{inclexcl1}) and (\ref{inclexcl2}), gives us
	\[\big|\#(\Z\cap\mathbb{V}_h(\mathbf{k};B))-\mu(\mathbb{V}_h(\mathbf{k};B))\big|\le 2^{ns+1}(n+1)\quad \forall \mathbf{k}\in \mathcal{K}_h(B)\]
    and thus 
\[	\big|\#(\Z\cap\A_{h}(f,S,\varepsilon,B,\gamma))-\mu(\A_{h}(f,S,\varepsilon,B,\gamma))\big|\le 2^{n(s+1)}n!^{s}(n+1).\]

If $S'=\emptyset$, then $\A_{h}(f,S,\varepsilon,B,\gamma)=\emptyset$ for all $h\in \N_{S}$ which do not divide $H_S(f)$. In this case we get, therefore, the bound
\[ \big|\#(\Z\cap\A(f,S,\varepsilon,B,\gamma))-\mu(\A(f,S,\varepsilon,B,\gamma))\big|\le 2^{n(s+1)}n!^{s}(n+1)\sigma_0(H_{S}(f)), \]
where $\sigma_0(H_{S}(f))$ denotes the number of (positive) divisors of $H_{S}(f)$. 

Let us now suppose that $s':=\#S'\ge 1$, and let $C>0$ be a constant such that $|f(x)|\le C (1+|x|)^n$ for all $x\in \R$. Clearly, $\Z\cap\A_{h}(f,S,\varepsilon,B,\gamma)=\emptyset$ for all $h\in \N_{S}$ with $h>C(1+B)^n$. Moreover, for any $h_0\in \N_S$ with $h_0|H_S(f)$, one has

\[ \begin{split}
    \underset{h'h_0>C(1+B)^n}{\sum_{h'\in\N_{S'}}}&\mu(\A_{h'h_0}(f,S,\varepsilon,B,\gamma))=\underset{h'>Ch_0^{-1}(1+B)^n}{\sum_{h'\in\N_{S'}}}C_{h_0}(f)\mu(\A_{h'}(f,S',\varepsilon,B,\gamma h_0))
    \\
    &\ll_{f,S'} \Big(\underset{h'>Ch_0^{-1}(1+B)^n}{\sum_{h'\in\N_{S}}}h^{-1/R_{S'}(f)}
    \Big) C_{h_0}(f) B\\
    & \ll_{f,S'} C_{h_0}(f) h_0^{1/R_{S'}(f)} (C(1+B)^n)^{-1/R_{S'(f)}} B \log(Ch_0^{-1}(1+B)^n)^{s'-1}
    \\
    &\ll_{f,S'} C_{h_0}(f) h_0^{1/R_{S'}(f)} B^{1-n/R_{S'(f)}} \log(B)^{s'-1}
    \\
    & \ll_{f,S'} C_{h_0}(f) h_0^{1/R_{S'}(f)} (\log B)^{s'-1}\quad \text{as $B\to\infty$.}
\end{split} \]
Summing over the (positive) divisors of $H_S(f)$, we get then
\[\begin{split} \underset{h>C(1+B)^n}{\sum_{h\in\N_{S}}}&\big|\#(\Z\cap\A_{h}(f,S,\varepsilon,B,\gamma))-\mu(\A_{h}(f,S,\varepsilon,B,\gamma))\big|
\\
&=\underset{h>C(1+B)^n}{\sum_{h\in\N_{S}}}\mu(\A_{h}(f,S,\varepsilon,B,\gamma))
\\
&\ll_{f,S'} \Big(\sum_{h_0|H_S(f)} C_{h_0}(f) h_0^{1/R_{S'}(f)}\Big)  (\log B)^{s'-1} \text{as $B\to\infty$.}
\end{split}\]

On the other hand, using the obvious bound 
\[ \#\{h'\in \N_{S'}\,:\,h'h_0\le C(1+B)^n\}\le \log(C(1+B)^n)^{s'}\]
for all $h_0\in \N_S$ with $h_0|H_S(f)$, we see that
\[  \begin{split} \underset{h\le C(1+B)^n}{\sum_{h\in\N_{S}}}&\big|\#(\Z\cap\A_{h}(f,S,\varepsilon,B,\gamma))-\mu(\A_{h}(f,S,\varepsilon,B,\gamma))\big|
\\
&\le 2^{n(s+1)}n!^{s}(n+1)\sigma_0(H_{S}(f))\log(C(1+B)^n)^{s'}
\\
&\ll_{f,S'} n!^{s-s'}\sigma_0(H_S(f)) (\log B)^{s'}\quad \text{as $B\to \infty$},
\end{split}\]
which concludes the proof.
\end{proof}

For $\gamma=1$, proposition \ref{errorbound} tells us that 
    \begin{equation} \label{ewe}
        \widetilde{N}(f,S,\varepsilon,B)=\mu(\A(f,S,\varepsilon,B,1))+\mathcal{O}_{f,S,\varepsilon}((\log B)^{s'})\quad \text{as $B\to \infty$},
    \end{equation}
 which, combined with proposition \ref{leadinggen}, proves theorem \ref{maintheorem}. 
 
 \begin{remark}
 Note that (\ref{ewe}) also holds when $S'=\emptyset$, in which case it tells us that $\widetilde{N}(f,S,\varepsilon,B)=\mathcal{O}_{f,S,\varepsilon}(1)$ as $B\to \infty$. However, this is trivial, because from section \ref{Igusasec} we know that if $S'=\emptyset$ then there exists $H\in \N_{S}$ such that $[f(x)]_S\le H$ for all $x\in \Z$. It follows that all $x\in \Z$ such that $|f(x)|^{\varepsilon}\le [f(x)]_{f,S}$ must satisfy $|f(x)|\le \xi_f(H)^{1/\varepsilon}$, and there are only finitely many integer $x$ for which this can be true. This of course implies that if $S'=\emptyset$ then for all $B$ big enough (depending on $f,S,\varepsilon$) one has
    \[\widetilde{N}(f,S,\varepsilon,B)=\#\{x\in\Z\,:\,|f(x)|^{\varepsilon}\le [f(x)]_{f,S}\}<\infty.\]
\end{remark}

\section{Proof of theorem \ref{goalunivariate}} \label{proofthm1}

    To the setting of the previous section, we add now the assumption that $f$ has no multiple roots in $\Z_p$ for any $p\in S'$. Since the set $S$ is in this case trivially $f$-balanced, theorem \ref{maintheorem} tells us that as long as $s':=\#S'\ge 1$ one has 
    \[ N(f,S,\varepsilon,B)\asymp_{f,S,\varepsilon} B^{1-n\varepsilon} (\log B)^{s'-1}\quad \text{as $B\to \infty$} \]
    for all $\varepsilon\in (0,1/n)$.
    
    The goal of this section is to show that the limit
    \begin{equation} \label{limit}
        \lim_{B\to \infty}  \frac{N(f,S,\varepsilon,B)}{B^{1-n\varepsilon}(\log B)^{s'-1}}
    \end{equation}
    exists if and only if $s'\ge 2$, which is the content of theorem \ref{goalunivariate}.
    
    By proposition \ref{ancom}(c), we have that for all $p$ for which $f$ has a root in $\Z_p$ one has
    \[  \mu_p(U_{p^k}(f))=\mu_p(U_{p^{a_p(f)+1}}(f))\cdot p^{-(k-a_p(f)-1)} \quad \forall k\ge a_p(f)+1,\]
    with $a_p(f)$ as in definition \ref{fundquant}(2), and thus
    \[\mu(\A(f,\{p\},\varepsilon,B,\gamma))=\mu_p(U_{p^{a_p(f)+1}}(f)) \cdot \mathcal{U}(f,\{p\},\varepsilon,B,\gamma p^{a_p(f)+1},-1)+\mathcal{O}_{f,p,\gamma,\varepsilon}(1)\]
    as $B\to\infty$, for all $\gamma\in \R_{>0}$,

    If $S=S'=\{p\}$, then this, together with proposition \ref{errorbound}, implies that
    \[N(f,\{p\},\varepsilon,B)=\mu_p(U_{p^{a_p(f)+1}}(f)) \cdot\mathcal{U}(f,\{p\},\varepsilon,B,p^{a_p(f)+1},-1)+\mathcal{O}_{f,S,\varepsilon}(\log B) \]
    as $B\to\infty$. By proposition \ref{U}(a), we get thus
    \[ \liminf_{B\to\infty} \frac{N(f,\{p\},\varepsilon,B)}{B^{1-n\varepsilon}}
    =2\cdot\mu_p(U_{p^{a_p(f)+1}}(f))p^{a_p(f)+1} \cdot \lambda^{-}(n,-1,p,\varepsilon)\cdot |c_f|^{-\varepsilon} \]
    and
    \[ \limsup_{B\to\infty} \frac{N(f,\{p\},\varepsilon,B)}{B^{1-n\varepsilon}}
    =2\cdot\mu_p(U_{p^{a_p(f)+1}}(f))p^{a_p(f)+1} \cdot  \lambda^{+}(n,-1,p,\varepsilon)\cdot |c_f|^{-\varepsilon}, \]
    which shows that the limit (\ref{limit}) does not exist (cf. definition \ref{lambdas}).
    
    In the case $S\supsetneq S'=\{p\}$, proposition \ref{errorbound} tells us similarly that
  \[ \begin{split} 
  N(f,S,\varepsilon,B)=&\mu_p(U_{p^{a_p(f)+1}}(f))\sum_{h_0|H_S(f)} C_{h_0}(f) \cdot\mathcal{U}(f,\{p\},\varepsilon,B,h_0p^{a_p(f)+1},-1)
  \\
  &+\mathcal{O}_{f,S,\varepsilon}(\log B)\quad \text{as $B\to \infty$}.
  \end{split}\]

    The non-existence of the limit (\ref{limit}) can proved in this case by working out the analogues of the results in section \ref{sums} that led to the proof of the  non-existence of the limit (\ref{limit}) in the case $S=S'=\{p\}$. However, the oscillation is now more complicated to describe, and the actual (quite tedious) computation is not too enlightening. For this reason, we prefer to omit it.

    \vspace{0.3cm}
    
    Let us now suppose $s'\ge 2$. Then, by proposition \ref{errorbound}, we have
    \[ \begin{split}
        N(f,S,\varepsilon,B)&=\mu(\A(f,S,\varepsilon,B,1))+\mathcal{O}_{f,S,\varepsilon}((\log B)^{s'})
        \\
        &=\sum_{h_0|H_S(f)} C_{h_0}(f) \cdot\mu(\A(f,S',\varepsilon,B,h_0))+\mathcal{O}_{f,S,\varepsilon}((\log B)^{s'})
    \end{split}\]
    as $B\to\infty$. Moreover, for any $\gamma\in \R_{>0}$, propositions \ref{ancom} and \ref{U} give us 
        \[ \begin{split}
            \bigg|\mu(\A(f,S',\varepsilon,B,\gamma))&-\Big(\prod_{p\in S'} \mu_p(U_{p^{a_p(f)+1}}(f))\Big)\mathcal{U}\Big(f,S',\varepsilon,B,\gamma \prod_{p\in S'} p^{a_p(f)+1},-1\Big)\bigg| 
            \\
            &= \sum_{p\in S'} \sum_{k=0}^{a_p(f)}\sum_{h\in \N_{S'\setminus \{p\}}} \mu(\A_{p^kh}(f,S',\varepsilon,B,\gamma))
            \\
            &\ll_{f,S',\varepsilon}   \sum_{p\in S'} \sum_{k=0}^{a_p(f)}\sum_{h\in \N_{S'\setminus \{p\}}} \mu_{\infty}(V_f(B,(\gamma p^k h)^{1/\varepsilon})) (p^{k}h)^{-1}
            \\
            &=\sum_{p\in S'} \sum_{k=0}^{a_p(f)} p^{-k} \mathcal{U}(f,S',\varepsilon,B,\gamma p^k,-1)
            \\
            &\ll_{f,S',\varepsilon} \gamma \cdot B^{1-n\varepsilon}(\log B)^{s'-2}\quad \text{as $B\to \infty$,}
        \end{split}
        \]
    with implied constants independent of $\gamma$, and thus

        \[  \mu(\A(f,S',\varepsilon,B,\gamma))\sim \frac{2 c(S')}{1-n\varepsilon}\Big(\prod_{p\in S'} \mu_p(U_{p^{a_p(f)+1}}(f))p^{a_p(f)+1}\Big)\cdot \gamma \cdot B^{1-n\varepsilon} (\log B)^{s-1} \]
    as $B\to \infty$, by proposition \ref{U}(b).

     Therefore, we arrive to
    \[ N(f,S,\varepsilon,B)\sim C(f,S,\varepsilon)\cdot B^{1-n\varepsilon}(\log B)^{s'-1}\quad \text{as $B\to \infty$},\]
    with
    \[ C(f,S,\varepsilon):=\frac{2c(S')}{1-n\varepsilon}\Big(\sum_{h_0|H_S(f)} C_{h_0}(f) h_0\Big)\Big(\prod_{p\in S'}\mu_p(U_{p^{a_p(f)+1}}(f))p^{a_p(f)+1}\Big),\]
    which concludes the proof of theorem \ref{goalunivariate}.

    \begin{remark}
        If $f\in\Z[X]$ is a polynomial of degree $n\ge 2$ and discriminant $\Delta(f)\ne 0$, then for all $p\in S'$ one can replace $a_p(f)$ with $v_p(\Delta(f))$ in the above formula for $C(f,S,\varepsilon)$. Indeed, it is an immediate consequence of \cite{Stewart}*{Theorem 2} that $\mu(U_{p^k}(f))p^k=\mu(U_{p^{v_p(\Delta(f))+1}}(f))p^{v_p(\Delta(f))+1}$ for all $k\ge v_p(\Delta(f))+1$. Under the additional assumption that the leading coefficient of $f$ be invertible in $\Z_p$, an easy application of Krasner's lemma tells us that $a_p(f)\le v_p(\Delta(f))$. To see this, let $K_p$ be a splitting field of $f$ over $\Q_p$ and let $\alpha_1,\dots,\alpha_n\in \mathcal{O}_{K_p}$ be the roots of $f$ in $K_p$, with $\alpha_1,\dots,\alpha_l\in \Z_p$ and $\alpha_{l+1},\dots,\alpha_n\not\in \Z_p$ for some $l\in \{1,\dots,n-2\}\cup \{n\}$.
        If $l=n$, then one has
        \[ a_p(f)=n \lambda_p(f) \le n(n-1)\lambda_p(f)\le v_p(\Delta(f)), \]
        where the last inequality follows immediately from the definition of $\lambda_p(f)$.
        
        Suppose now that $l\le n-2$, and let $g(X):=(X-\alpha_{l+1})\dots (X-\alpha_n)$. If $x\in \Z_p$ and $i\in \{l+1,\dots,n\}$, then by Krasner's lemma there exists $j\in \{l+1,\dots,n\}$ distinct from $i$ such that $|x-\alpha_i|_p\ge |\alpha_j-\alpha_i|_p$. It follows that
        \[|x-\alpha_i|_p\ge \prod_{j\in \{l+1,\dots,n\}\setminus\{i\}} |\alpha_j-\alpha_i|_p\quad \forall x\in \Z_p\]
        and thus
        \[ |g(x)|_p\ge \prod_{i=1}^{l}\prod_{j\in \{l+1,\dots,n\}\setminus\{i\}} |\alpha_j-\alpha_i|_p=|\Delta(g)|_p \quad \forall x\in \Z_p, \]       
        which shows that $u_p(g)\le v_p(\Delta(g))$.
     
        If $l=1$, then we have 
        \[ \begin{split}
            a_p(f)&=\lambda_p(f)+u_p(g)
            \\
            &\le 2(n-1) \lambda_p(f)+v_p(\Delta(g))
            \\
            &\le 2v_p(g(\alpha))+v_p(\Delta(g))
            \\
            &=v_p(\Delta(f)).
        \end{split} \]
        
        Finally, in the case $2\le l\le n-2$ we get
        \[ \begin{split} a_p(f)&=l \lambda_p(f)+u_p(g) 
            \\
                &\le l(l-1)\lambda_p(f)+v_p(\Delta(f))
            \\
                &\le v_p(\Delta(f/g))+v_p(\Delta(f))
            \\
                &\le v_p(\Delta(f)),
        \end{split}\]
        which concludes the proof.
    \end{remark}

\section*{Acknowledgements}

Most of the research work behind this paper has been performed in the context of the author's master's thesis. The author is extremely grateful to his master's thesis advisor Dr. Jan-Hendrik Evertse for the suggestion of the topic and all the helpful tips. 

This research did not receive any specific grant from funding agencies in the public, commercial, or not-for-profit sectors.


\normalsize

\begin{bibdiv}
\begin{biblist}

\baselineskip=17pt

\bib{Preprint}{article}{
    title =        {S-parts of values of univariate polynomials, binary forms and decomposable forms at integral points},
    author =       {Bugeaud, Y.},
    author = 	   {Evertse, J.-H.},
    author =	   {Gy\H{o}ry, K.},
    journal =      {Acta Arith.}, 
    volume=		   {184},
    year =         {2018},
    pages = 	   {151--185},
    DOI=            {10.4064/aa170828-7-3} 
}


\bib{Beta}{article}{
    author =       {Everest, G. R.},
    title =        {Uniform distribution and lattice point counting},
    journal =      {J. Austral. Math. Soc.},
    series = 	   {A},
    volume = 	   {53},
    pages =        {39--50},
    year =         {1992},
    DOI=            {10.1017/S1446788700035370}
}


\bib{AnalComb}{book}{
	author =       {Flajolet, P.}, 
	author =       {Sedgewick, R.},
    title =        {Analytic Combinatorics},
    publisher=     {Cambridge University Press},
    year =         {2009}
}

\bib{History}{article}{
    author =       {Gross, S.},
    author =       {Vincent, A.},
    title =        {On the factorization of $f(n)$ for $f(x)$ in $\Z[x]$},
    journal =      {Int. J. Number Theory}, 
    volume =       {9},
    pages=         {1225--1236},
    year =         {2013},
    DOI=            {10.1142/S179304211350022X}
}

\bib{Igusabook}{book}{
    author = {Igusa, J.-I.},
    title = {An introduction to the theory of local zeta functions},
    series = {AMS/IP Studies in Advanced Mathematics},
    volume = {14},
    publisher = {Amer. Math. Soc.},
    year = {2000}
}

\bib{Koblitz}{book}{
	author=        {Koblitz, N.},
    title =        {$p$-adic Numbers, $p$-adic Analysis, and Zeta-Functions (Second Edition)},
    publisher=     {Springer-Verlag New York, Inc.},
    year=          {1984}
}

\bib{LiuThesis}{thesis}{
    author =       {Liu, J.},
    title =        {On $p$-adic Decomposable Form Inequalities},
    type =         {PhD Thesis},
    place =        {Leiden University Repository},
    year =         {2015}
}


\bib{Thesis}{thesis}{
    author =       {Moreschi, M.},
   title =        {S-parts of values of univariate polynomials and decomposable forms},
    type =         {Master's Thesis},
    place =      {Leiden University Repository}, 
   year =         {2018}
}

\bib{Neukirch}{book}{
	author=        {Neukich, J.},
    title =        {Algebraic Number Theory},
    publisher=     {Springer-Verlag New York, Inc.},
    year=          {1999}
}


\bib{Stewart}{article}{
    author =       {Stewart, C. L.},
    title =        {On the number of solutions of polynomial congruences and Thue equations},
    journal =      {J. Amer. Math. Soc.},
    volume=        {4(4)},   
    pages =        {793--835},
    year =         {1991},
    DOI=            {10.1090/S0894-0347-1991-1119199-X}
}

\end{biblist}
\end{bibdiv}
\end{document}